\documentclass[12pt]{article}
\usepackage[centertags]{amsmath}
\usepackage{amsfonts}
\usepackage{amssymb}
\usepackage{amsthm}
\input{amssym.def}

\newtheorem{Definition}{Definition}[section]
\newtheorem{Theorem}[Definition]{Theorem}
\newtheorem{Lemma}[Definition]{Lemma}

\newcommand{\lc}{\mathcal{L}}
\newcommand{\rc}{\mathcal{R}}
\newcommand{\hc}{\mathcal{H}}
\newcommand{\jc}{\mathcal{J}}

\title{\Large \bf Nil extensions of simple regular ordered semigroup}
\author{A. K. Bhuniya  and K. Hansda \\
\footnotesize{Department of Mathematics, Visva-Bharati
University,}\\
\footnotesize{Santiniketan, Bolpur - 731235, West Bengal, India}\\
\footnotesize{anjankbhuniya@gmail.com}, \
\footnotesize{kalyanh4@gmail.com}}

\begin{document}

\maketitle

\begin{abstract}{\footnotesize}
In this paper, nil extensions of some special type of ordered
semigroups, such as,  simple regular ordered  semigroups, left simple
and right regular ordered semigroup. Moreover, we have characterized  complete semilattice decomposition of all ordered semigroups which are nil extension  of ordered semigroup.
\end{abstract}

\section{Introduction:}
Nil  extensions of a semigroup(without order), are precisely the
ideal extensions  by a nil semigroup. In 1984, S. Bogdanovic and S.
Milic \cite{bc1} characterized the semigroups(without order) which
are nil extensions of completely simple semigroups, where as, a
similar work was done by J. L Galbiati and M.L Veronesi\cite{GV} in
1980. Authors like S. Bogdanovic, M. Ciric, Beograd have
investigated this type extensions for regular semigroup, group,
periodic semigroup as well as completely regular semigroup(see
\cite{bc2}, \cite{bco2}).

 The concepts of nil  extensions have been extended to
ordered semigroups by Y. Cao in the paper \cite{Cao 2000} with
characterizing nil extensions  all type of simple ordered
semigroups. Cao was the first to do so.  The notion of ideal
extensions in ordered semigroups is actually introduced  by  N.
Kehayopulu and M. Tsingelis in \cite{ke2003}. In \cite{Ke2009} they
have  worked give on semigroups which are nil extensions of
Archimedean ordered semigroups.

This paper inspired by the results of \cite{bc2} \cite{bco2} \cite{bc3}. The aim of this paper is to describe all those ordered semigroups
which are  the nil extension of   simple
regular, left simple right regular and completely regular ordered semigroups.

Our paper organized as follows. The basic definitions and properties
of ordered semigroups  are presented in Section 2. Section 3 is
devoted to characterizing the nil extensions of simple regular, left
simple right regular and completely regular ordered semigroups. In
Section 4 we have described complete semilattice decomposition of
all ordered semigroups which are nil extension  of  ordered
semigroups.

\section{Preliminaries:}
In this paper $\mathbb{N}$ will provide the set of all natural
numbers. An ordered semigroup is a partiality ordered set $S$, and
at the same time a semigroup $(S,.)$ such that $( \forall a , \;b ,
\;x \in S ) \;a \leq b \Rightarrow  \;xa\leq xb \;and \;a x \leq b
x$. It is denoted by $(S,., \leq)$. For an
ordered semigroup $S$ and $H \subseteq S$, denote $$(H]_S := \{t \in
S : t \leq  h, \;for \;some \;h\in  H\}.$$ Also it is denoted by
$(H]$ if there is no scope of confusion.

The  relation $|$ on an ordered semigroup $S$ is defined by $$a|b
\Leftrightarrow  \;there \;exists \;x, y \in S^1 \;such \;that  \;b
\leq xay.$$
Let $I$ be a nonempty subset of an ordered semigroup $S$. $I$ is a
left (right) ideal of $S$, if $SI \subseteq I \;( I S \subseteq I)$
and $(I]= I$. $I$ is an ideal of $S$ if $I$ is both a left and a
right ideal of $S$. An (left, right) ideal $I \;of S$ is proper if
$I \neq S$. $S$ is left (right) simple if it does not contain proper
left (right) ideals. $S$ is  simple if it does not contain any
proper ideals. The intersection of all ideals of an ordered
semigroup $S$ , if nonempty, is called the kernel of $S$ and is
denoted by $K(S)$.

An ordered semigroup $S$ is called a group like ordered semigroup or
GLO-semigroup if for all $a, b \in S \;there \;are \;x, y \in S
\;such \;that \;a \leq xb \;and \;a \leq by$. An ordered semigroup
$S$ is called a left group like ordered semigroup or LGLO- semigroup
if for all $a, b \in S \;there \;are \;x \in S$ such that $a \leq
xb$ Similarly we define right group like ordered semigroup or
RGLO-semigroup. A band is a semigroup in which every element is an
idempotent. A commutative band is called a semilattice. Kehayopulu
\cite{Ke2006} defined Greens relations on an ordered semigroup $S$
as follows: $  a \lc b   \; if   \;L(a)= L(b),
  \;a \rc b   \; if   \;R(a)= R(b),  \;a \jc b   \; if   \;I(a)= I(b), \;and \;\hc= \;\lc \cap \;\rc$.
These four relations $\lc, \;\rc, \;\jc \;and \;\hc$ are equivalence
relations. In an  ordered semigroup $S$, an equivalence relation
$\rho$ is called left (right) congruence if for $a, b, c \in S \;a
\;\rho \;b \; implies \;ca \;\rho \;cb \;(ac \;\rho \;bc)$. $\rho$
is congruence if it is both left and right congruence. A congruence
$\rho$ on $S$ is called semilattice if  for all $a, b \in S \;a
\;\rho \;a^{2} \;and \;ab \;\rho \;ba$. A semilattice congruence
$\rho$ on $S$ is called complete if $a \leq b$ implies $(a,ab)\in
\rho$. The ordered semigroup $S$  is called complete semilattice of
subsemigroup of type $\tau$ if there exists a complete semilattice
congruence $\rho $ such that $(x)_{\rho}$ is a type $\tau$
subsemigroup of $S$. Equivalently: There exists a semilattice $Y$
and a family of subsemigroups $\{S\}_{\alpha \in Y}$ of type $\tau$
of $S$ such that:
\begin{enumerate}
\item \vspace{-.4cm}
$S_{\alpha}\cap S_{\beta}= \;\phi$ for any $\alpha, \;\beta \in \;Y
\;with \; \alpha \neq \beta,$
\item \vspace{-.4cm}
$S=\bigcup _{\alpha \;\in \;Y} \;S_{\alpha},$
\item \vspace{-.4cm}
$S_{\alpha}S_{\beta} \;\subseteq \;S_{\alpha \;\beta}$ for any
$\alpha, \;\beta \in \;Y,$
\item \vspace{-.4cm}
$S_{\beta}\cap (S_{\alpha}]\neq \phi$ implies $\beta \;\preceq
\;\alpha,$ where $\preceq$ is the order of the semilattice $Y$
defined by \\$\preceq:=\{(\alpha,\;\beta)\;\mid
\;\alpha=\alpha\;\beta\;(\beta\;\alpha)\}$ \cite{ke1}.
\end{enumerate}

Before going to the main results of this paper we shall illustrate
some important  definition and results related to  nil extensions of
ordered semigroups.

\begin{Lemma}\cite{Cao 2000}
Let $I$ be an ideal of an ordered semigroup $S$. Then $I = K(S)$ if
and only if $I$ is  simple.
\end{Lemma}

\begin{Definition}\cite{Cao 2000}
 Let $S$ be an ordered semigroup with zero $0 (i.e., \;there
\;exists \;0 \in S \;such \;that \;0x = x0 = 0 \;for \;any \;x \in S
)$ . An element $a \in S$ is called a nilpotent if $a^n \leq 0$ for
some $n \in \mathbb{N}$. The set of all nilpotents of $S$ is denoted
by $Nil(S)$. $S$ is called nil ordered semigroup (nilpotent) if $S =
Nil(S)$.
\end{Definition}

\begin{Lemma}\cite{ke2003}
 Let $(S, .,\leq)$ be an ordered semigroup and $K$ an
ideal of $S$. Let  $S/K := (S\backslash K) \;\cup \{0\}$, where $0$
is an arbitrary element of $K \;(S\backslash K$ is the complement of
$K \;to \;S$). Define an operation $" * "$ and an order $" \preceq "
\;on \;S/K$ as follows:
 $$* : S/K \times S/K \longrightarrow S/K \;| (x, y) \longrightarrow (x
 * y) \;where$$
 \[ x * y :=\left \{\begin{array}{l l}
 xy &  \mbox{if}~~~~ \;xy \in  S/K \\
 0  &  \mbox{if}~~~~ \;xy \in K
 \end{array}\right .\]

$\preceq := (\leq \cap [(S/ K) \times (S/ K)]) \cup \{(0, x)| x \in
S/K\}$. Then $(S/K, *,\preceq)$ is an ordered semigroup and $0$ is
its zero.
\end{Lemma}
If there is no confusion, we denote the multiplication and the order
in different ordered semigroups by same symbol.

We know that $(S/K, *)$ is a Rees semigroup where the congruence is
the Rees Congruence on $S$. Now we shall proceed to define the ideal
extension of ordered semigroup in Kehayopulu's way.
\begin{Definition}\cite{ke2003}
Let $(S, \cdot,\leq_S)$ be an ordered semigroup, $(Q, \cdot,\leq_Q)$
an ordered semigroup with $0, \;S\cap Q^*= \phi, \;where \; Q^*=
Q\setminus \{0\}$. An ordered semigroup $(V, \cdot, \leq_V)$ is
called an ideal extension of $S$ by $Q$ if there exists an ideal
$S'$ of $V$ such that $(S', \cdot, \leq_{S'})\approx (S, \cdot,
\leq_{S})$ and $(V/ S', *, \preceq)\approx (Q, \cdot,\leq_Q)$, where
$\leq_{S'}=\leq_V \cap (S'\times S')$ and $"*", \;"\preceq"$ the
multiplication and the order on $V/ S'$ defined in \ref{1}
\end{Definition}
So,  now we are ready to consider  the definition of nil extension
of ordered semigroup given by Cao.
\begin{Definition}\cite{Cao 2000}
Let $I$ be an ideal of an ordered semigroup $S$. Then $(S/I,.,
\preceq)$ is called the Rees factor ordered semigroup of $S \;modulo
\;I$, and $S$ is called an ideal extension of $I$ by ordered
semigroup $S/I$. An ideal extension $S \;of \;I$ is called a
nil-extension of $I$ if $S/I$ is a nil ordered semigroup.
\end{Definition}
\begin{Lemma}\cite{Cao 2000}
 Let $S$ be an ordered semigroup and $I$ an ideal of $S$. Then
the following are equivalent:
\begin{itemize}
\item[(i)] $S$ is a nil-extension of $I$;
\item[(ii)] $(\forall a\in S)(\exists m \in \mathbb{N}) \;a^m \in I$.
\end{itemize}
\end{Lemma}
In {\cite{bh1}, we have  introduced the notion of Clifford and left
Clifford ordered semigroups and characterized their structural
r
\begin{Lemma}
Let $S$ be an ordered semigroup. Then
\begin{enumerate}
  \item \vspace{-.4cm}
 for $a \in Reg_\leq(S)$, there is $x \in Reg_\leq(S)$ such that $a \leq
axa$.
 \item \vspace {-.4cm}
for $a \in \mathbf{R}Reg_\leq(S)$ and any $k \in \mathbb{N}, \;a^k
\in (a^{k+n} S]$ for all $n \in \mathbb{N}$.
\end{enumerate}
\end{Lemma}
\section{nil extension of completely regular ordered semigroups:}
\begin{Theorem}
An ordered semigroup $S$ is nil extensions of left simple and right
regular ordered semigroup if and only if  for all $\;a,b \in S,
\;there \;exists \;n \in \mathbb{N} \;such \;that \;a^{n}\in
(a^{2n}Sb] $ and for all $a \in S, \;and \;b \in
\mathbf{R}Reg_\leq(S)$ such that $a\leq ba$ implies $a\leq a^{2}x
\;for \;some \;x\in S$.
\end{Theorem}
\begin{proof}
Suppose  $S$ is a nil extension of a left simple and right regular
ordered  semigroup $K$. Let $a,b \in S$. Then there exists $\;m \in
\mathbb{N}$ such that $a^{m}, a^mb \in K$. Since $K$ is a right
regular ordered semigroup, there is $x \in K$ such that $a^{m} \leq
a^{2m} x$. The left simplicity of $K$ yields $x \leq t'a^{m}b$ for
some $t' \in K$. So $a^m \in (a^{2m} S b] $. Again let $b \in
\mathbf{R}Reg_\leq (S)$ and $a\in S$ such that $a \leq ba$. Since $b
\in \mathbf{R}Reg_\leq (S), \;there \;exists \;z \in K$, such that
$b\leq b^{n+1}(z)^{n}$, which holds for all $n\in\mathbb{N}$. So
there is  some $n_{1} \in \mathbb{N} \;such \;that \;z^{n_{1}} \in
K$. Since $K$ is an ideal of $S$, $b^{n_{1}+1}(z)^{n_{1}} \in K$.
Thus $b \in K$, so $ba\in K$ and hence  $a\in K$. Then the right
regularity of $K$ yields $a \in (a \in a^2S]$. Thus the condition is
necessary.

Conversely, let us assume that the given condition holds in $S$.
Choose $a\in S$ arbitrarily. Then by the given condition $\;there
\;exists \;m \in \mathbb{N} \;and \;y \in S$ such that $a^{m}\leq
a^{2m}yb$. Clearly $\mathbf{R}Reg_\leq (S)\neq \phi$. Say
$T=\mathbf{R}Reg_\leq(S)$. It is evident that for each $b\in S,
\;there \;exists \;k \in \mathbb{N}$ such that $b^{k}\in T$.
Consider  $s\in S$ and $a\in T$. Then there exists $y\in S\;such
\;that \;a \leq a^{2}y$. That is  $a \leq a^{3}y^{2} \leq
........................ \leq a^{n+1}(y)^{n},\; \;for \;all \;n\in
\mathbb{N}$. Then $sa \leq saa^{n}y^{n}, \;for \;all \;n \in
\mathbb{N} $.  From the given condition there are $m_1 \in
\mathbb{N} \;and \; t_1 \in S$ such that $y^{m_1} \leq y^{2m_1} t_1
sa$, whence  $sa\leq sa(a^{m}y^{2m_{1}}t_1) sa$. Set $t=
a^{m}y^{2m_{1}}t_1$, Then $sa \leq sat sa$. Now $sat \leq (sat)^2
sat$, which shows that $sat \in T$. Using the given second condition
to $sa \in S \;and \;sat \in T; \; sa  \leq sat  sa $ yields that
$sa \in ((sa)^2 S]$, that is, $sa \in T$.

In a similar approach it can be easily obtain that $as \leq a^{m_2
+1} z^{2m_2}t_2a as$ for some $m_2 \in \mathbb{N}$. Now since
$a^{m_2 +1} z^{2m_2}t_2 \in S$, by preceding conclusion $a^{m_2 +1}
z^{2m_2}t_2a \in S$. Hence it follows given condition that, $as \leq
(as)^2 a^{m_2 +1} z^{2m_2}t_2a as$, that is $as \in ((as)^2 S]$.
Therefore $sa, \;as \in T$. And so $T$ is a right regular
subsemigroup of $S$.

Consider $a\in S$ and $b\in T$ such that $a \leq b$. Then for some
$u \in \mathbb{N}, \;a\leq b^{n+1}u^{n}, \;for \;all \;n \in
\mathbb{N}$. Then there exists $m_{3}\in \mathbb{N}$  such that $a
\leq (b^{m_{3}+1}u^{m_{3}}a) a$. Since $b^{m_{3}+1}u^{m_{3}}a \in
T$, the given condition yields that $a \leq a^{2}t_{4}\;for \;some
\;t_{4}\in S$.

 Finally, choose $c,d \in T$. Then $c \leq c^{n+1}v^{n}
\;for \;all \;n \in \mathbb{N}$ and for some $v \in S$. Now for $v,
d \in S, \;there \;exists \;t_{5} \in S$ such that $v^{m_4}\leq
v^{2m_{4}}t_{5}d$ for some $m_{4}\in \mathbb{N}$. Since $c \in T, \;
c^{m_{4}}t_{5}\in T$ and so $c \in (T d]$. Hence $T$ is left simple.
Thus $T$ is left simple and right  regular ordered semigroup. Hence
$S$ is nil extension of a left simple and right  regular ordered
semigroup $T$.
\end{proof}

In \cite{Cao 2000} authors have characterized ordered semigroup
which are nil extensions of simple(order) semigroups. In the next
result we have described ordered semigroups which are nil extensions
of both simple and ordered regular semigroups.
\begin{Theorem}
An ordered semigroup $S$ is nil extensions of  simple and regular
ordered semigroup if and only if the following conditions are holds
\begin{itemize}
\item[(i)]for all $a,b \in S$, there exists $n \in \mathbb{N}$ such
that $a^{n} \in (a^{n}SbSa^{n}] $.
\item[(ii)] for  $a \in S$, and $b
\in Reg_\leq(S)$  such that   $a\leq ba$ implies that  $a
\in(aSbSa]$..
\item[(iii)] for  $a \in S$, and $b
\in Reg_\leq(S)$  such that   $a\leq ab$ implies that  $a \in
(aSbSa]$.
\item[(iv)] for  $a \in S$, and $b
\in Reg_\leq(S)$  such that   $a\leq b$ implies that  $a \in
Reg_\leq(S)$.
\end{itemize}
\end{Theorem}
\begin{proof}
Suppose that $S$ is  a nil extension of a simple and regular ordered
semigroup $K$.

$(i)$ Choose $a,b \in S$. Then there exists $m \in \mathbb{N}$ such
that $a^{m},a^{m}b \in K$. Since $K$ is regular ordered semigroup,
there is some $x \in K$ such that $a^{m} \leq a^{m}xa^{m} \leq
a^{m}(xa^{m})xa^{m}$, whence, by the simplicity of $K$ there is\\
$y\in K$ such that $a^{m}\leq a^{m}(ya^{m})bxa^{m}\in
(a^{m}SbSa^{m}]$.

$(ii)$ Let $b \in Reg_\leq (S)$ and $a \in S$ such that $a \leq ba$.
Since $b \in Reg_\leq(S), \;there \;exists \\\;z \in S$ such that $b
\leq b(zb)^{n} \;for \;all \;n \in \mathbb{N}$. Then for some $n_{1}
\in \mathbb{N}$ such that $(zb)^{n_{1}} \in K$. This gives
$b(zb)^{n_{1}} \in K$, which gives that $b \in K$ and so $ba \in K$
and finally $a\in K$. Since $K$ is a simple and regular ordered
semigroup, for $a,b\in K, \;a \in (aKbKa]\subseteq (aSbSa]$.
\\$(iii)$  and $(iv) $ follows similarly.

Thus the given conditions are obtained.

Conversely, let us assume that given condition holds in $S$. Choose
$a\in S$  arbitrarily. By the given condition, there is
$m\in\mathbb{N}$ and some $x \in S$ such that $a^{m}\leq
a^{m}xa^{m}$. So $Reg_\leq(S) \neq \phi$. Say $T=Reg_\leq(S)$. It is
evident that, for every $a\in S, \;there \;exists \;m \in
\mathbb{N}$ such that $a^{m}\in T$. Consider  $s\in S$ and $a\in T$.
Then for some $h\in S, \;a \leq aha$, whence $sa \leq
sa(ha)^{n},\;and \;as \leq (ah)^n as \;for \;all \;n \in
\mathbb{N}$. Then there is $m_1 \in \mathbb{N}$ such that
$(ha)^{m_1} \in T$. So using the condition(ii) to $sa \leq sa
(ha)^{m_1}$ yields that $sa \in ((sa) S (ha)^{m_1} S sa]$, that is
$sa \in T$. On the otherhand $as \in T$ follows similarly by using
condition(ii). Hence by condition(iv) $T$ is an ideal of $S$. Also
by the given condition(iv) it can be easily seen that $T$ is simple
in $S$. Hence $S$ is nil extension of a simple and  regular ordered
semigroup $T$.
\end{proof}
\begin{Theorem}
An ordered semigroup $S$ is nil extensions of completely
regular ordered semigroup if and only if S is completely $\pi$-regular and
for all $a\in S, \;b\in Reg_\leq(S)$ such that  $a \leq ba$ and $a
\leq ab$, both separately implies $a\leq (a^{2}SbSa^{2}] \;for
\;some \;x\in S $.
\end{Theorem}
\begin{proof}
Suppose that $S$ is a nil extension of an completely regular
ordered semigroup $K$ and $a,b \in S$. Then  there is $ \;m \in
\mathbb{N}$ such that $a^{m}\in K$. By the complete regularity of
$K$,it is possible to get some  $m\in \mathbb{N}$ and $x\in K$ such
that $a^{m} \leq a^{2m}xa^{2m} \leq a^{m+1}xa^{m+1}$ Thus $S$ is
completely $\pi$-regular.

Again let $b \in Reg(S)$ and $a \in S$ such that $a \leq ba$ and $a
\leq ab$. Since $b \in Reg(S), \;there \;is \;z\in S$ such that
$b\leq b(zb)^{n} \;for \;all \;n \in \mathbb{N}$. Then there is
$n_{1}\in \mathbb{N}$ such that $(zb)^{n_{1}} \in K$. That is
$b(zb)^{n_{1}} \in K$, which gives that $b\in K$ and so $ba\in K$
and finally $a\in K$. Since $K$ is an ordered completely regular
ordered semigroup, by theorem ()\cite{bh1}, it follows that $K$ is
complete semilattice of ordered completely simple semigroups and
$\jc$ is a complete semilattice congruence on $K$. So for $a, ba \in
K \;such \;that  \; a \leq ba \;yields
 \;(a)_{\jc}=(ba)_{\jc}$.
Thus  $a\in (a^{2}KbKa^{2}] \subseteq (a^{2}SbSa^{2}]$. Thus the
given conditions are obtained.

Conversely, let us assume that given condition holds in $S$. Let
$a\in S$ be arbitrary. Since $S$ is completely $\pi$-regular, there
exists $m\in\mathbb{N}$ such that $a^{m}\leq a^{m+1}xa^{m+1}$, for
some $x\in S$. So $Gr_\leq(S)\neq \phi$. Say $T=Gr(S)$. Also it is
clear that, for each $a \in S, \;there \;exists \;m\in \mathbb{N}$
such that $a^{m}\in T$. Now choose $s\in S$ and $a\in T$. Then by
lemma(), there exists $h \in S\;such \;that \;a \leq a^{2}ha^{2}
\leq a(ah)a^{2} \leq a(a^{2}h)^{2}a^{2} \leq a^{2}(a^{2}h)^{3}a^{2}
\leq a^{n-1}(a^{2}h)^{n}a^{2},\; \;for \;all \;n \in \mathbb{N}$.
Now $\exists m_{1}\in \mathbb{N}$ such that
$a^{n-2}(a^{2}x)^{n}a^{2}\in T$. Again $sa \leq
sa^{n-1}(a^{2}h)^{n}a^{2} =    saa^{n-2}(a^{2}h)^{n}a^{2},\;for
\;all \;n  \in \mathbb{N}$. Then using the given condition and the
preceding observation $sa \leq saa^{n-2}(a^{2}h)^{n}a^{2} \leq satsa
\;where  \;t=  s'a^{n-2}(a^{2}h)^{m_{1}}a^{2}t'\in S \; and  \;for
\;some\;s',t'\in S$. Now consider $a \in S$ and $b\in T$ such that
$a\leq b$. Then by the given condition, $a \leq at_{1}bt_{2}a \;for
\;some \;t_{1}, t_{2}\in S$. Also $T$ is ordered completely regular
semigroup such that $a^{m}\in T \;for \;all \;a \in S$ and for some
$m \in \mathbb{N}$. Hence $S$ is nil extension of a
completely regular ordered semigroup $T$.
\end{proof}
\section{complete semilattice of nil extensions simple ordered regular semigroups:}
\begin{Theorem}
Let $S$ be an ordered semigroup. Then the following conditions are
equivalent on $S$:
\begin{itemize}
\item[(i)] $S$ is complete semilattice  of nil extensions simple
ordered  regular semigroups;
\item[(ii)]  for all $a, b \in S$ there is $n \in \mathbb{N}$ such
that $(ab)^n \in ((ab)^n S a^2 S (ab)^n]$;
\item[(iii)] $S$ is $\pi-$regular, and is  complete semilattice of
Archimedean semigroups;
\item[(iv)] $S$ is $\pi-$regular, and for all $a \in S$,  $e \in E_{\leq}(S)$, $a \mid e$ implies $a^2 \mid
e$;
\item[(v)] $S$ is $\pi-$regular, and for all $a, b \in S$, and  $e \in E_{\leq}(S)$, $a \mid e \;and \;b \mid e$ implies $ab \mid
e$;
\item[(vi)] $S$ is $\pi-$regular, and every $\jc-$class of $S$
containing an ordered idempotent is a subsemigroup of $S$;
\item[(vii)]$S$ is intra $\pi-$regular and every $\jc-$class of $S$
containing an intra-regular element  is a regular subsemigroup of
$S$;
\item[(viii)]$S$ is complete semilattice  of nil extensions simple
semigroups and $Intra_\leq(S)= Reg_\leq (S)$.
\end{itemize}
\end{Theorem}
\begin{proof}
$(i)\Rightarrow (ii)$: Suppose that $S$ is complete semilattice $Y$
of semigroups $S_\alpha(\alpha \in Y)$, where each  $S_\alpha$ is
nil extensions of a simple ordered regular semigroup $K_\alpha$. Let
$\rho$ be the corresponding complete semilattice congruence on $S$.
Let $a, b \in S$. Then there are $\alpha, \beta \in Y$ such that
$a\in S_\alpha$ and $b \in S_\beta$, so that there is $\gamma \in Y$
such that $ab, a^2b \in S_\gamma$. Let $S_\gamma$ be the nil
extensions of $K_\gamma$. Then there are $m, n \in \mathbb{N}$ such
that $(ab)^m, (a^2b)^n \in K_\gamma$. Since $K_\gamma$ is ordered
regular, $(ab)^m \leq (ab)^m z (ab)^m$ for some $z \in K_\gamma$.
Then for $z, (a^2b)^n \in K_\gamma$, the simplicity of $K_\gamma$
yields that $z \in (K_\gamma (a^2b)^n K_\gamma]$. Hence $(ab)^m \in
((ab)^m S a^2 S(ab)^m]$.

$(ii)\Rightarrow (iii)$ and  $(iii)\Rightarrow(iv)$: These  follow
from Theorem 2.8 \cite{ke1}

 $(iv)\Rightarrow (v)$: This  follows from
Lemma 2.6.

$(v)\Rightarrow (vi)$: Consider a $\jc-$class $J$ of $S$ that
contains an ordered idempotent $e \in E_\leq (S)$. Choose $a, b \in
J$ arbitrarily. Then there are $x, y, z, w \in S^1$ such that $e
\leq xa y$ and $e \leq zbw$. Now it  is obvious that $ab \in (SeS]$.
Moreover  by the given condition, $ab \mid e$. Thus $ab \jc e$.
Hence $J$ is subsemigroup of $S$.

$(i)\Rightarrow (vii)$: Suppose that $S$ is complete semilattice $Y$
of semigroups $S_\alpha(\alpha \in Y)$, where each  $S_\alpha$ is
nil extensions of simple ordered regular semigroup $K_\alpha$.
Moreover let $\rho$ be the corresponding complete semilattice
congruence on $S$. Consider $a \in S$ then there are $\alpha \in Y$
and $m \in \mathbb{N}$ such that $a^m \in K_\alpha$, where $S_\alpha
$ is nil extensions of a simple ordered regular semigroup
$K_\alpha$. So by the ordered regularity of $K_\alpha$ there is $x
\in K_\alpha$ such that $a^m \leq a^mxa^m$. Since $x, a^{2m} \in
K_\alpha$; the simplicity of $K_\alpha$ yields that $a^m \leq a^m y
a^{2m} z a^m$ for some $y,z \in K_\alpha$. This shows that $S$ is
intra $\pi-$regular.

Next consider a $\jc-$class $J \;of \;S$ that contains a
intra-regular element $a$. Then $a \leq xa^2y $ for some $x, y \in
S^1$ and there is some $\alpha \in Y$ such that $a \in S_\alpha$.
 Now $a \leq xa^2y \leq xa x a^2 y^2$ implies that
$(a)_\rho =(xay)_\rho= (xa)_\rho$, so $xa \in S_\alpha$. Likewise,
$ay \in S_\alpha$. It is evident that $a \leq (xa)^n ay^n \;for
\;all \;n \in \mathbb{N}$. For $xa \in S_\alpha$ there is $m' \in
\mathbb{N}$ such that $(xa)^{m'} \in K_\alpha$, so $(xa)^{m'}
ay^{m'} \in K_\alpha$, that is, $a \in K_\alpha$.

Now it  is obvious  that $a \jc a^2$. Choose $b, c \in J$. Since $a
\jc b \;and \;a \jc c$ there are $z_1, w_1, x_1, y_1 \in S^1$ such
that $a \leq x_1 b y_1$ and $a \leq z_1 c w_1$. This implies that
$a^2 \leq z_1 (c w_1x_1b) y_1$, which together with the completeness
of $\rho $ gives that $(a)_\rho=(a^2)_\rho =( az_1 (c w_1x_1b)
y_1)_\rho= ( az_1 (c w_1x_1b)^m y_1)_\rho$ for all $m \in
\mathbb{N}$. Then for all $m \in \mathbb{N}, \;az_1 (c w_1x_1b)^m
y_1 \in S_\alpha $. Since $S_\alpha$ is nil extension of $K_\alpha$
for  some $n' \in \mathbb{N}$ we have that $(az_1 (c w_1x_1b)^m
y_1)^{n'} \in K_\alpha$. Then for $a, (az_1 (c w_1x_1b)^m y_1)^{n'}
\in K_\alpha$ the simplicity of $K_\alpha $ yields that $a \leq t_1
(az_1 (c w_1x_1b)^m y_1)^{n'} t_2$ for some $t_1, t_2 \in K_\alpha$.
This shows that $a \leq s_1(c w_1x_1b)^m s_2$ for some $s_1, s_2 \in
S$ and for all $m \in \mathbb{N}$. Since $ S$ is intra
$\pi-$regular, there are $u_1, u_2 \in S$ and $m \in \mathbb{N}$
such that $a \leq s_1 u_1 (c w_1x_1b)^{2m} u_2 s_2$, that is, $a
\leq s_1 u_1  (c w_1x_1b)^2(c w_1x_1b)^{2m-2} u_2 s_2 \leq s_1 u_1 c
w_1x_1(bc) w_1x_1b (c w_1x_1b)^{2m-2} u_2 s_2 \leq z_1 bc z_2$,
where $z_1= s_1 u_1  c w_1x_1$ and $z_2= w_1x_1b (c w_1x_1b)^{2m-2}
u_2 s_2$. Also, it is evident that $bc \in (S^1aS^1]$. Thus  $a \jc
bc$, which  shows that $J$ is subsemigroup of $S$. So what we have
seen is that every element of $J$ is intra-regular. Also we have
seen earlier that, $a$ being an intra-regular element $a \in
K_\alpha$. This clearly shows that $J= K_\alpha$. Hence every $\jc-$
class of $S $ containing an intra-regular element is an ordered
regular subsemigroup of $S$.

$(i)\Rightarrow (iii)$:  Let $S$ be complete semilattice $Y$ of
semigroups $\{S_\alpha \}_{\alpha \in Y}$ and $\rho$ be the
corresponding complete semilattice congruence on $S$.  Choose $a, b
\in S_\alpha$ such that $a \mid b$. Then there are $x, y \in S^1$
such that $b \leq xay$. Now there is $n, m  \in \mathbb{N}$ such
that $a^n, b^m \in K_\alpha$, where $S_\alpha$  is nil extension of
simple regular semigroup $K_\alpha$. Since $a^n, b^m \in K_\alpha$,
the simplicity of $K_\alpha$ implies that $b^m \leq z_1 a^n z_2$ for
some $z_1, z_2 \in K_\alpha \subseteq S_\alpha$. Hence $S_\alpha$ is
an Archimedean ordered semigroup. This shows that $S$ is complete
semilattice of Archimedean ordered semigroups.

Also it is obvious that $S$ is $\pi-$regular.

$(i)\Rightarrow (viii)$: Suppose that  $S$ is   a complete
semilattice of semigroups $\{S_\alpha \}_{\alpha \in Y}$ and $\rho$
be the corresponding  complete semilattice congruence on $S$. Let
for $\alpha \in Y, \;S_\alpha$ is nil extension of simple ordered
regular semigroup $K_\alpha$. Here we only have to show that
$Intra_\leq (S)= Reg_\leq (S)$. For this, consider $a \in Reg_\leq
(S)$. Then there is $\alpha \in Y$ such that $a \in S_\alpha $ and
there is $s \in S_\alpha$ such that $ a \leq asa$. Also by the
completeness of $\rho, \;(a)_\rho= (sa)_\rho$, so $sa \in E_\leq
(S_\alpha)$. Then there is some $m' \in \mathbb{N}$ such that
$(sa)^{m'} \in K_\alpha$.  Since $K_\alpha $ is an ideal of
$S_\alpha \;and \;for \;all \;n \in \mathbb{N} \;a \leq a (sa )^n $,
it follows that $a \in K_\alpha$. So finally  the simplicity and the
ordered  regularity of $K_\alpha$ yields that $a \in Intra_\leq
(S)$.

Conversely, let $b \in Intra_\leq (S)$. Then there are  $x, y \in
S^1$  and $\gamma \in Y$ such that $b \leq xb^2 y$ and  $b \in
S_\gamma$. Now for all $n \in \mathbb{N}, b \leq (xb)^n by^n$. Also
by completeness of $\rho$ we have that $(b)_\rho= (xby)_\rho= (xb xb
by^2)_\rho= (xb xby)_\rho= (xb)_\rho (xby)_\rho= (xb)_\rho (b)_\rho=
(xb)_\rho$. Since $xb \in S_\gamma$, there is $m^{''} \in\mathbb{N}$
such that $(xb)^{m^{''}} \in K_\gamma$, where $S_\gamma$ is  nil
extension of simple ordered regular semigroup $K_\gamma$. Thus $b
\in K_\gamma$. Therefore  by the ordered regularity of $K_\gamma$ it
follows that $a \in Reg_\leq (S)$. Hence $Intra_\leq (S)= Reg_\leq
(S)$.

$(viii)\Rightarrow(i)$ Suppose that  $S$ is a complete semilattice
of semigroups $S_\alpha (\alpha \in Y)$, where $S_\alpha$ is nil
extension of simple semigroup $K_\alpha$. Let $a \in K_\alpha$. Then
by the simplicity of $K_\alpha, \;a \in Intra_\leq (S)= Reg_\leq (S)
$. This implies that $a \leq axa \leq a (xa)^n xa $ for some $x \in
S$ and for all $n \in \mathbb{N}$. The completeness of $\rho$ gives
that $(a)_\rho= (xa)_\rho$. Thus  $xa \in S_\alpha$ and so for some
$m\in \mathbb{N}, \;(xa)^m \in K_\alpha$, that is $(xa)^m x \in
K_\alpha$. Hence $K_\alpha$ is ordered regular and so $S$ is
complete semilattice of nil extensions simple ordered regular
semigroups.
\end{proof}

\begin{Theorem}
The following conditions are equivalent on an ordered semigroup $S$:
\begin{itemize}
\item[(i)] $S$ is complete semilattice  of nil extensions of left group like ordered
semigroups;
\item[(ii)] $S$ is complete semilattice  of nil extensions of left simple and right regular ordered
semigroups;
\item[(iii)] $S$ is complete semilattice $Y$ of nil extensions of completely regular ordered  semigroups
$K_\alpha(\alpha \in Y)$ and every $\lc-$class of $S$ that contains
a completely regular element $a \;(a \in S_\alpha)$ is $K_\alpha$;
\item[(iv)]$S$ is completely $\pi-$regular  and is a complete semilattice  of left Archimedean ordered semigroups.
\item[(v)]for all $a, b \in S$ there is $n \in \mathbb{N}$ such
that $(ab)^n \in ((ab)^n aSb(ba)^n]$;
\item[(vi)] $S$ is $\pi-$regular, and for all $a \in S$,  $e \in E_{\leq}(S)$, $a \mid e$ implies $a
\mid_l e$;
\item[(vii)]$S$ is complete semilattice $Y$ of nil extensions of completely regular ordered  semigroups
$K_\alpha(\alpha \in Y)$ and every $\lc-$class of $S$ that contains
an ordered  idempotent $e \;(e \in S_\alpha)$ is $K_\alpha$;
\item[(viii)] $S$ is complete semilattice of nil extensions of completely regular semigroups $K_\alpha$ and
every $\lc-$class of $S$ that contains an  ordered regular element
$a \;(a \in S_\alpha)$ is $K_\alpha$;
\item[(ix)]$S$ is complete semilattice  of nil extensions left simple
semigroups and $LReg_\leq(S)= Reg_\leq (S)$;
\item[(x)] $S$ is complete semilattice  of nil extensions left
Clifford ordered  semigroups $K_\alpha$ and every $\lc-$class of $S$
containing a regular element $a \;(a \in S_\alpha )$ is  $K_\alpha$.
\end{itemize}
\end{Theorem}
\begin{proof}
$(ii)\Rightarrow (i)$: Suppose that $S$ is complete semilattice $Y$
of semigroups $\{S_\alpha\}_{\alpha \in Y}$ and let $S$ be a nil
extension of a left simple and right regular semigroup $K_\alpha$.
Choose  $a \in K_\alpha$. Since $K_\alpha$ is right regular,  there
is  some $x \in K_\alpha$ such that $ \;a \leq a^2 x$. For $ax, a
\in K_\alpha$, the simplicity of $k_\alpha$ yields that  $a \in
(aK_\alpha a]$. Thus $K_\alpha $ is ordered regular. Hence $S$ is
nil extension of left group like ordered semigroup.

$(ii)\Rightarrow (iii)$: Let $S$ be complete semilattice of
semigroups $\{S_\alpha\}_ {\alpha \in Y}$ and let $S_\alpha$ be a
nil extension of a left simple and right regular semigroup
$K_\alpha$. Let $a \in K_\alpha$. Since $K_\alpha$ is right regular,
$a \in (a^2K_\alpha]$. Also as  $K_\alpha$ is left simple we have
that  $a \in (a^2 K_\alpha a^2]$. Thus  $K_\alpha$ is a completely
regular ordered semigroup.

Next, consider a $\lc-$class $L$ of $S$ that contains a complete
regular element $a$, where $a \in S_\alpha$ for some $ \alpha \in
Y$. By Lemma 2.8 it follows that $a \in K_\alpha$ and $L \subseteq
K_\alpha$.

Now to show that $K_\alpha \subseteq L$ let us choose $y \in
K_\alpha$. By the simplicity of $K_\alpha$ we have that $y \leq za$
for some $z \in K_\alpha$. Since $za \in K_\alpha$ it follows that $
\;y \in K_\alpha$.

$(iii)\Rightarrow (iv)$:  Suppose that $S$ is complete semilattice
$Y$ of semigroups $\{S_\alpha\}_ {\alpha \in Y}$. Let $a \in S$.
Then there is $\alpha \in Y$ such that $a \in S_\alpha$. Let
$S_\alpha$ be a nil extension of completely regular ordered
semigroup $K_\alpha $. Then there is $m \in \mathbb{N}$ such that
$a^m \in (a^{2m} S_\alpha a^{2m}] \subseteq (a^{m+1} S a^{m+1}]$.
This shows that $S$ is completely $\pi-$regular ordered semigroup.

Now choose $a, b \in S_\alpha$. Then there is $m' \in \mathbb{N}$
such that $a^{m'}, a^{m'}b \in K_\alpha$. Consider the $\lc-$class
$L_{a^{m'}}$ of $a^{m'} \in S$. Since $a^{m'} \in Reg_(S)$, using
the given condition we have that $L_{a^{m'}}= K_\alpha$. So for
$a^{m'}, a^{m'}b \in L_{a^{m'}}, \;a^{m'} \in (K_\alpha b] \subseteq
(S_\alpha b]$. Hence $S$ is complete semilattice of left Archimedean
semigroups.

$(iv)\Rightarrow (v)$: This is  obvious.

$(iv)\Rightarrow (vi)$: Let $S$ satisfies the given condition. Then
clearly $S$ is $\pi-$regular. Let $\rho$ be the complete semilattice
congruence on $S$. Choose $a \in S$ and $e \in E_\leq (S)$ such that
$a \mid e$. Then there are $x, y \in S^1$ such that $e \leq xay $.
This implies that $(e)_{\rho}= (exya)_{\rho}$. Since $(e)_{\rho}$ is
left Archimedean and $e, \;exya \in (e)_{\rho}$, there is $m^{''}
\in \mathbb{N}$ such that $e^{m^{''}} \in ((e)_{\rho} exya]$. Thus
$e \in (Sa]$. That is $a \mid_{l} e$

$(iv)\Rightarrow (vi)$ Let $L$ be $\lc-$class of $S$ containing an
ordered idempotent $e \in S_\alpha $. Then by lemma(2.8), $e \in
K_\alpha \;and \;K_\alpha \subseteq S_\alpha $. Again by the
simplicity of $K_\alpha, K_\alpha \subseteq L$. Hence $L= K_\alpha$.

$(iii)\Leftrightarrow(viii)$:  Here we just have to show the later
part. Let consider an $\lc-$class $L$ of $S$, so that $L$ contains a
regular element $a \in S_\alpha$. Then there is $m \in \mathbb{N}$
such that $a^m \in K_\alpha$. By the regularity of $a$, there is $x
\in S$ such that $a \leq a (xa)^n$ for all $n \in \mathbb{N}$. So
for some $l \in \mathbb{N}, \;(xa)^l \in K_\alpha$. Now $a^m \in
Gr_\leq(S)$, and $K_\alpha $ is the $\lc-$class of $S$ containing
the complete regular element $a^m$, which also contains the regular
element $a$.

Converse is obvious.

$(viii)\Rightarrow(ix)$: Suppose that $S$ is complete semilattice
$Y$ of semigroups $S_\alpha \;(\alpha \in Y)$ and $S_\alpha $ is a
nil extension of complete  regular semigroup $K_\alpha$. Let $\rho$
be the corresponding complete semilattice.  Consider $a, b \in
K_\alpha$ and $L_a, \;L_b$ are the $\lc-$classes of $a, b $
respectively. Using the given condition we have $L_a= K_\alpha=
L_b$. This yields that $a \in L_b$, that is, $a \in (Sb]$. Since
$K_\alpha$ is completely regular ordered semigroup, there is $y \in
K_\alpha$ such that $b \leq byb$. Also $a \in L_b$ implies that $a
\leq xb$ for some $x \in S$. So $a \leq (xby) b$. Then by the
completeness of $\rho$ we have that
$(a)_\rho=(axbyb)_\rho=(xbyb)_\rho= (xb)_\rho$. Therefore $xb \in
S_\alpha$. Since $K_\alpha$ is an ideal of $S_\alpha, \;xby \in
K_\alpha$. Hence $a \in (K_\alpha b]$ so that  $K_\alpha $ is left
simple. Therefore $S$ is complete semilattice of  a nil extension of
left simple semigroups.

Now consider $a \in LReg_\leq (S)$. Then for all $n \in \mathbb{N}$
$ \;a \leq xa^2 \leq..........\leq x^n a^{n+1}$ for some  $x \in
\S$. Then for some $ l\in \mathbb{N}, \;x^l a^{l+1} \in K_\alpha$
and since $K_\alpha $ is an ideal of $S_\alpha$,  $a \in K_\alpha$.
Thus $LReg_\leq (S)\subseteq Reg_\leq (S)$. On the other hand,
choose $b \in Reg_\leq (S)$. Then $b \in S_\beta$ for some $\beta
\in Y$, where $S_\beta$ is nil extension of left simple semigroups
$K_\beta$. Since $b \in Reg_\leq(S) $, there is $z \in S$ such that
$b \leq b (zb)^n$ for all $n \in \mathbb{N}$. Now by the
completeness of $\rho, \;zb \in S_\beta$, which follows that
$(zb)^{m_1} \in K_\beta$ for some $m_{1} \in \mathbb{N}$. So $b \in
K_\beta$. Since $K_\beta$ is left simple, $b \in (K_\beta b^2]
\subseteq (Sb^2]$. Thus $Reg_\leq (S)= LReg_\leq (S)$. This
completes the proof.

$(ix)\Rightarrow(i)$: Let $S$ be  complete semilattice $Y$ of left
simple semigroups $K_\alpha \;(\alpha \in Y)$ and $\rho $ be the
corresponding complete semilattice congruence on $S$. To show that
each $K_\alpha(\alpha \in Y)$ are left group like ordered
semigroups, first choose $a \in K_\alpha$. Then by the left
simplicity of $K_\alpha, \;a \in (Sa^2]$. Therefore $a \in LReg_\leq
(S)= Reg_\leq(S)$. And so there is $y \in S$ such  that $a \leq aya$
and obviously $a \leq a (ya)^n ya$ for all $n \in \mathbb{N}$. Now
by the completeness of $\rho, \;(a)_\rho= (a^2ya)_\rho= (ya)_\rho$,
so $ya \in S_\alpha$. Thus for some $m \in \mathbb{N}, \;(ya)^m \in
K_\alpha$ and so $(ya)^my \in K_\alpha$. This implies that $a \leq a
(ya)^m y a$, where $(ya)^m y\in K_\alpha $. Therefore $K_\alpha$ is
ordered regular subsemigroup. Also $K_\alpha$ is left simple. Hence
$K_\alpha$ is left group like ordered
semigroup.

$(iii)\Rightarrow(ii)$: Let $S$ be a complete semilattice $Y$ of
semigroups $\{S_\alpha\}_{\alpha \in Y}$ and let each $S_\alpha$ are
nil extension of complete regular ordered semigroup $K_\alpha$.

Consider $a \in K_\alpha$ and $L_a$ be the $\lc-$class of $a \in S$.
Since $a \in Gr_\leq(S)$ by the given condition we have that  $L_a=
K_\alpha$. This implies that $K_\alpha$ is left simple. Also, by the
complete regularity of $a,$ it is evident that $a$ is right regular.
Hence $S $ is complete semilattice congruence of left simple and
right regular ordered semigroups.

$(vii) \Rightarrow( iii)$: Here we are only to show that every
$\lc-$class of $S$ that contains a completely regular element $a$
(where $a \in S_\alpha$) is $K_\alpha$. For this let us consider a
$\lc-$class $L$ of $S$ that contains a completely regular element
$a$ where $a \in S_\alpha$ for some $\alpha \in Y$.  Since $a \in
Reg_\leq(S)$ we have that $a \in K_\alpha$, by Lemma 2.8. Since
$K_\alpha$ is a completely regular ordered semigroup there is  $h
\in K_\alpha$ such that $a \leq aha$. Clearly $ha \in E_\leq(S)$.
Also $a  \lc ha$. Therefore $ha \in L$. Thus $L$ is a $\lc-$class of
$S$ that contains an ordered idempotent $ha \in S_\alpha$. Therefore
by the given condition $L= K_\alpha$.

$(i)\Rightarrow(ix)$:  First suppose that $S$ is complete
semilattice $Y$ left group like ordered semigroups $K_\alpha (\alpha
\in Y)$. So $S$ is complete semilattice of nil extension of left
simple ordered semigroups $K_\alpha$.

Now choose $a \in LReg_\leq(S)$. By Lemma 2.8, it follows that $a
\in K_\alpha$, that is $a \in Reg_\leq(S)$.

Conversely, let $a \in Reg_\leq (S)$. By Lemma 2.8, it is evident
that $a \in K_\alpha$. Then by the simplicity of $K_\alpha, \;a\in
(K_\alpha a^2] \subseteq (S a^2]$. Hence $LReg_\leq(S)= Reg_\leq
(S)$.

$(i)\Leftrightarrow(x)$: Let $S$ be a  complete semilattice $Y$ left
group like ordered semigroups $K_\alpha (\alpha \in Y)$ and $\rho $
be the corresponding complete semilattice congruence on $S$. Choose
$a, b \in K_\alpha$. Then for $ab, a \in K_\alpha$, the left
simplicity of $K_\alpha$, yields that $ab \in (K_\alpha  a]$. Hence
$S$ is complete semilattice of left Clifford ordered semigroups.

Now, let $L$ be a $\lc-$class  of $S$ containing a regular element
$a$. Then by Lemma 2.8, $a \in K_\alpha $. Choose $y \in K_\alpha$.
Then  by the left simplicity of $K_\alpha$ together with $a \in
K_\alpha$ yields that $y \lc a, \;so  \;y \in L$, that is,
$L\subseteq K_\alpha$. Also by Lemma 2.8, $L \subseteq K_\alpha$.
Hence the condition  follows.

Conversely, let $S$ be complete semilattice $Y$ of nil extension of
left Clifford ordered semigroups $K_\alpha(\alpha \in Y)$.let $a,b
\in K_\alpha$. Since  $K_\alpha$ is a left Clifford ordered
semigroup, $ \;ab \leq sa$, for some $s \in K_\alpha$. Next consider
the $\lc-$class $L_a$ that contains $a$, obviously then $L_a=
K_\alpha$. Thus $ab, a \in L_a$. So $a \lc ab$. Then there is $z \in
S$ such that $a \leq zab$. Since $a \in K_\alpha, \;za \in
K_\alpha$. This implies that $K_\alpha$ is left simple. Hence $S$ is
complete semilattice of left group like ordered semigroup.
\end{proof}

\bibliographystyle{plain}

\end{document}